\documentclass[12pt,oneside]{amsart}

\usepackage{amssymb}
\usepackage{amsmath}
\usepackage{amsthm}
\usepackage{amscd}

\usepackage{longtable}

\theoremstyle{plain}
\newtheorem{theorem}{Theorem}
\newtheorem{lemma}{Lemma}

\newtheorem{proposition}{Proposition}

\theoremstyle{definition}

\theoremstyle{remark}
\newtheorem{remark}{Remark}

\begin{document}
\title
[Asymptotic estimates of holomorphic sections]
{Asymptotic estimates of holomorphic sections on Bohr-Sommerfeld Lagrangian submanifolds
}
\author{Yusaku Tiba}
\date{}

\begin{abstract}
Let $M$ be a complex manifold and 
$L$ be a line bundle over $M$ with a Hermitian metric $h$ whose Chern form is a K\"ahler form $\omega$.  
Let $X \subset M$ be a compact Lagrangian submanifold of $(M, \omega)$.  
When $X$ satisfies the Bohr-Sommerfeld condition, 
we give an asymptotic estimate of the norm $|f|_{h^k}$ on $X$ for $f \in H^0(M, L^k)$.  
\end{abstract}

\maketitle



\section{Introduction}\label{section:1}
Throughout this paper $M$ will be a K\"ahler manifold of dimension $n$ equipped with a K\"ahler form $\omega$ and a complex structure $J \in \mathrm{End}(TM)$.   
The bundle $L$ is a holomorphic prequantum line bundle over $M$.  
That is, 
$L$ is a holomorphic line bundle with a Hermitian metric $h$ such that the Chern form $c_{1}(L, h)$ associated with $h$ equals $\omega$.  
We denote by $\nabla$ the Chern connection of $(L, h)$.  
We consider the $k$-th tensor power $L^k$ of $L$.   
Let $f, g$ be sections of $L^k$. 
We denote by $\langle f, g \rangle_{h^k}$ the pointwise scalar product and 
define the integral product
$
(f, g)_{h^k} = \int_{M} \langle f, g \rangle_{h^k} \omega_{n}
$
where $\omega_{n} = \omega^n/n!$.  
We write $|f|^2_{h^k} = \langle f, f \rangle_{h^k}$ and $\|f\|^2_{h^k} = (f, f)_{h^k}$.  
Let $L^2(M, L^k)$ be the Hilbert space of square integrable sections of $L^k$.  
We define $H^{0}_{(2)}(M, L^k) = H^{0}(M, L^k) \cap L^{2}(M, L^k)$.  
This space is regarded as the quantum phase space of $X$ with the Planck constant $h=1/k$.  
Letting $k$ tend to infinity corresponds to letting $h$ tend to 0, which is referred to as the \textit{semiclassical limit}.
The asymptotic results as $k \to \infty$ expected to recover the laws of classical mechanics.  
The Bergman kernel $K_{k}(x, y)$ of $H_{(2)}^0(M, L^k)$ is the reproducing kernel for the Hilbert space $H^{0}_{(2)}(M, L^k)$, that is, $f(x) = (f(\cdot), K(x, \cdot))_{h^k}$ for any $f \in H^0_{(2)}(M, L^k)$ and $x \in M$.  
It is a well known property that the Bergman kernel function $B_{k}(x) = |K_{k}(x, x)|_{h^k}$ is characterized by 
\[
B_{k}(x) = \sup_{f \in H^0_{(2)}(M, L^k),\, f \neq 0} \frac{|f(x)|^2_{h^k}}{\|f\|^2_{h^k}}.  
\]
The asymptotic behavior, as $k \to +\infty$ of the Bergman kernel function is studied in detail, and the asymptotic series expansion formula of $B_{k}(x)$ is proved when $M$ is projective 
(see \cite{Tia}, \cite{Cat}, \cite{Zel}).  
Berndtsson~\cite{Ber} gave a simple proof for the leading order term 
\begin{equation}\label{equation:1}
B_{k}(x) \sim k^n \quad (k \to +\infty).  
\end{equation}
In this paper, we estimate holomorphic sections in $H^0_{(2)} (M, L^k)$ on a Bohr-Sommerfeld Lagrangian submanifold and provide an analogous result of (\ref{equation:1}).  

Let $X \subset M$ be a Lagrangian submanifold of $(M, \omega)$, that is, $X$ is a real $n$-dimensional submanifold of $M$ such that $\iota^{*}\omega = 0$.  
Here $\iota: X \to M$ is the inclusion map.  
Let $\nabla^{X}$ be the connection induced by $\nabla$ on $\iota^{*}L$.  
Then $(\iota^{*}L, \nabla^{X})$ is flat since $\iota^{*}\omega = 0$.  
We say that $(X, \nabla^{X})$ satisfies the Bohr-Sommerfeld condition if there exists a non-vanishing smooth section $\zeta \in C^{\infty}(X, \iota^{*}L)$ satisfying $\nabla^{X} \zeta = 0$ (cf.~\cite{Ioo}, \cite{Ver}).  
Hence $(X, \nabla^{X})$ satisfies the Bohr-Sommerfeld condition if and only if the holonomy of $(\iota^{*}L, \nabla^{X})$ is trivial.  
If we take $\zeta$ as $|\zeta|_{h} = 1$ on $X$, we call the data of $(X, \zeta)$ the Bohr-Sommerfeld Lagrangian submanifold.  
Let $d\mu_{X}$ be the Riemannian density induced by $\omega$ on $X$.  %
We define $\mathrm{Vol}(X, \omega) = \int_{X} d\mu_{X}$. 
The holomorphic section obtained from
the Bergman projection of the distribution $\zeta^k d\mu_{X}$ is regarded as the quantization of $X$.   The asymptotic behaviour of these sections as $k \to \infty$ has been extensively studied (cf. \cite{Bor-Pau-Uri}, \cite{Deb-Pao}, \cite{Ioo}).    
Our first result provides an asymptotic estimate that holds not only for such special holomorphic sections but also for any holomorphic section.
\begin{theorem}\label{theorem:1}
Let $X \subset M$ be a compact Lagrangian submanifold of $(M, \omega)$.  
Assume that $(X, \nabla^{X})$ satisfies the Bohr-Sommerfeld condition.  
Then 
\begin{equation}\label{equation:2}
\limsup_{k \to +\infty}\left(\frac{\mathrm{Vol}(X, \omega)}{(2k)^{n/2}} \sup_{f \in H^{0}_{(2)}(M, L^k), \, f \neq 0} \frac{\inf_{x \in X} |f(x)|^2_{h^k}}{\|f\|^2_{h^k}} \right) \leq 1.  
\end{equation}
\end{theorem}
We do not assume $M$ is projective or Stein in Theorem~\ref{theorem:1}.  
The next result shows that (\ref{equation:2}) is an optimal estimate in some cases.  
Let $A = \{a_1, a_2, \ldots, a_N\}$ be a finite sequence of points in $M \setminus X$ (possibly empty).  
We denote by 
$H^{0}_{(2), A}(M, L^k)$ the Hilbert space of holomorphic sections $f \in H^{0}_{(2)}(M, L^k)$ which has a zero at each point $a_{j}$ ($j = 1, \ldots, N$).  
Here, if $a \in M$ occurs $l$ times in $A$, then $f$ vanishes to order $l$ at $a$.  
\begin{theorem}\label{theorem:3}
Let $X \subset M$ be a compact Lagrangian submanifold of $(M, \omega)$.  
Assume that $(X, \nabla^{X})$ satisfies the Bohr-Sommerfeld condition.  
Let $A$ be a finite sequence of points in $M \setminus X$.  
We assume one of the following three conditions:  
\begin{itemize}
\item[(i)]
$M$ is a projective manifold.  
\item[(ii)]
$M$ is a Stein manifold and the Ricci form $Ric(\omega)$ of $\omega$ satisfies 
$Ric(\omega) \geq -C \omega$ on $M$ for some $C > 0$.   
\item[(iii)]
$M$ is a pseudoconvex domain in $\mathbb{C}^n$.  
\end{itemize}
Then 
\begin{equation}\label{equation:a}
\sup_{f \in H^{0}_{(2), A}(M, L^k), \, f \neq 0} \frac{\inf_{x \in X} |f(x)|^2_{h^k}}{\|f\|^2_{h^k}} \sim \frac{(2k)^{n/2}}{\mathrm{Vol}(X, \omega)} \quad (k \to +\infty).  
\end{equation}
\end{theorem}

One of our motivations for Theorem~\ref{theorem:3} is to study a quantitative version of the theorem due to \cite{Duv-Sib} and \cite{Gue}.  

\begin{theorem}\label{theorem:2}\text{$($Theorem~2.8 of \cite{Gue}$)$}
Assume that $M$ is a projective manifold.  
Let $X \subset M$ be a totally real submanifold and $\iota: X \to M$ be the inclusion map.  
The following are equivalent: 
\begin{itemize} 
\item[(a)] 
$X$ is rationally convex.  
\item[(b)]
There exists a smooth Hodge form $\theta$ for $M$ such that $\iota^{*} \theta = 0$.  
\end{itemize}
\end{theorem}
Here the condition (a) means that $M \setminus X$ is equal to a union of positive divisors of $M$.    
If $M = \mathbb{C}^n$, the polynomial convexity implies the rational convexity, and rational convex sets is an intermediary set between convex sets and polynomial convex sets.  
For any Lagrangian submanifold $X$ of $(M, \omega)$, 
there exists a positive integer $l$, a Hermitian metric $\tilde{h}$ of $L^l$ and its Chern connection $\widetilde{\nabla}$ of $L^l$ such that $(X, \widetilde{\nabla}^{X})$ satisfies the Bohr-Sommerfeld condition (see the remark after the proof of Proposition~\ref{proposition:1}).  
Hence, Theorem~\ref{theorem:3} implies $(b) \Rightarrow (a)$ of Theorem~\ref{theorem:2}.  
We note that Theorem~\ref{theorem:2} was first proved by \cite{Duv-Sib} in the case of $M = \mathbb{C}^n$.  

It would be an interesting problem whether the left hand side of (\ref{equation:a}) has the asymptotic series expansion.  
Bohr-Sommerfeld Lagrangian submanifolds and related asymptotic series expansion formulas are studied in symplectic settings (cf.~\cite{Ioo}).  
 
 In the proof of Theorem~\ref{theorem:1}, we use Demailly's Jensen-Lelong formula (cf.~Chapter~III of \cite{Dem}) with a potential function which satisfies the complex Monge-Amp\`ere equation outside $X$.  
There exists such a potential function if $X$ is a real-analytic manifold (\cite{Gui-Ste}).  
We note that similar computations appear in \cite{Ber-Ort}.  
In Section~\ref{section:4}, we reduce our situation to the real-analytic case. 

We introduce some notation.  
We use the notation $f \lesssim g$ to mean that $|f| \leq c |g|$ for some positive number $c$ which does not depend on $k$.  
We write $f_{k} = O(k^{-\infty})$ if $f_{k} \lesssim k^{-m}$ for any $m \in \mathbb{N}$.  
When $f(z_{1}, \ldots, z_{n})$ is a function on an open set in $\mathbb{C}^n$, 
we write $|\partial_{z} f| = (|\frac{\partial f}{\partial z_{1}}|^2 + \ldots + |\frac{\partial f}{\partial z_{n}}|^2)^{1/2}$, $|\overline{\partial}_{z} f| = (|\frac{\partial f}{\partial \overline{z}_{1}}|^2 + \ldots + |\frac{\partial f}{\partial \overline{z}_{n}}|^2)^{1/2}$.  
When $\sigma$ is an $L^k$-valued form, we denote by $|\sigma|_{h^k, \omega}$ (resp.~$\|\sigma\|_{h^k, \omega}$) the pointwise (resp.~integral) norm of $\sigma$ induced by $(h^k, \omega)$.  

\medskip 

{\it Acknowledgment.}
The author would like to express his gratitude to referees who informed him that our proof can be simplified by the complex microlocal analysis at least in the case of projective manifolds.  
Their suggestions lead the author to the study of non-projective cases.  
We are also grateful to another referee for their thorough reading of the manuscript, 
which greatly improved the clarity and exposition of the paper.
This work was supported by the 
Grant-in-Aid for Scientific Research (KAKENHI No.\! 21K03266).  
There are no conflicts of interests.  
All data generated or analysed during this study are included in this published article

\section{Bohr-Sommerfeld condition}\label{section:2}
Let $(M, \omega)$ be a K\"ahler manifold.  
Let $L \to M$ be a holomorphic prequantum line bundle over $M$ with the Chern connection $\nabla$.  
Let $X \subset M$ be a compact Lagrangian submanifold of $(M, \omega)$ and $\iota: X \to M$ be the inclusion map.  
We take a sufficiently small Stein tubular neighborhood $U \subset M$ of $X$.  
Since the de Rham cohomology class $[c_{1}(L, h)]$ vanishes on $U$, there exists $l \in \mathbb{N}$ such that $L|_{U}^l$ is a trivial holomorphic line bundle (cf.~\cite{Gue}). 
In this situation, we introduce a basic property of the Bohr-Sommerfeld condition.    
\begin{proposition}\label{proposition:1}
The following conditions are equivalent.  
\begin{itemize}
\item[(a)]
$(X, \nabla^{X})$ satisfies the Bohr-Sommerfeld condition.  
\item[(b)]
By shrinking $U$ if necessary, there exists a non-vanishing smooth section $s \in C^{\infty}(U, L)$ which satisfies $\nabla s = 0$ on $X$. 
\item[(c)]
By shrinking $U$ if necessary, 
there exists a non-vanishing smooth section $s \in C^{\infty}(U, L)$ such that $\log |s|^2_{h} = 0$ to order two on $X$ and $\nabla'' s = 0$ to order $m$ on $X$ for any $m \in \mathbb{N}$.  
\item[(d)]
By shrinking $U$ if necessary, 
there exists a non-vanishing holomorphic section $s_{0} \in H^0(U, L)$ which satisfies 
$\int_{\gamma}d^{c}\log |s_{0}|^2_{h} \in 4\pi \mathbb{Z}$ for any $\gamma \in H_1(X, \mathbb{Z})$.  
\end{itemize}
Here $\nabla''$ is the $(0, 1)$-part of $\nabla$ and $d^c$ is defined by $d^c f(v) = -df (Jv)$ for $v \in TM$.   
As $dd^c \log |s_0|_h^2 = -4\pi \omega$ in (d), and the restriction of $\omega$ to $X$ vanishes, we note that the restriction of $d^c \log |s_0|_h^2$ to $X$ is closed.  
\end{proposition}
\begin{proof}
Assume (a).  
There exists $\zeta \in C^{\infty}(X, \iota^{*}L)$ such that $\nabla^{X} \zeta = 0$ and $|\zeta|_{h} =1$ on $X$.  
This implies that $L|_{U}$ is a trivial smooth line bundle, and it is holomorphically trivial by the Oka principle.  
We take a non-vanishing holomorphic section $s_{0} \in H^0(U, L|_{U})$ 
and put $\varphi_{0} = - \log |s_{0}|^2_{h}$.  
For any $m \in \mathbb{N}$, a smooth function $\zeta/s_{0}$ on $X$ can be extended to a function $\xi \in C^{\infty}(U)$ with $\overline{\partial} \xi = 0$ to order $m$ on $X$ by the H\"ormander-Wermer Lemma (cf.~\cite{Hor-Wer}, Proposition~5.55 of \cite{Cie-Eli}).    
By taking $U$ sufficiently small, we may assume $\xi \neq 0$.  
Put $s = \xi s_{0}$.  
Then $\nabla'' s = 0$ to order $m$ on $X$.  
For any $p \in X$, we have $\log |s(p)|_{h}^2 = \log |\zeta(p)|^2_{h} = 0$, and 
\begin{equation}\label{equation:3}
\nabla s(p) = (d \xi - \xi \partial \varphi_{0})s_{0}(p) = (\partial \log (|\xi|^2e^{-\varphi_{0}}))s(p) = (\partial \log |s|^2_{h})s(p).  
\end{equation}
Since $\nabla_{v}s = \nabla^{X}_{v}\zeta = 0$ for any $v \in T_{p}X \subset T_pM$, 
we have 
\[
0 = \langle \partial \log |s|^2_{h}, v \rangle = \frac{1}{2} \langle d \log |s|^2_{h}, v \rangle + \frac{\sqrt{-1}}{2} \langle d^c \log |s|^2_{h}, v\rangle 
\]
where $\langle \cdot, \cdot \rangle$ denotes the natural pairing between $T_{p}M$ and $T^{*}_{p}M$.  
Hence, $\langle d \log |s|_h^2, v \rangle = 0$ and $\langle d^c \log |s|_h^2, v\rangle = -\langle d \log |s|_{h}^2, Jv \rangle = 0$.  
Since $T_{p}M = T_{p}X \oplus JT_{p}X$, we have  
$d \log |s|^2_{h}(p) = 0$ for any $p \in X$.  
This shows $(a) \Rightarrow (c)$.  
Moreover, $d \log |s|^2_{h}(p) = 0$ implies 
$\partial \log |s|^2_{h}(p) = 0$ for any $p \in X$, and (\ref{equation:3}) shows 
$(a) \Rightarrow (b)$.  

On the other hand, assume $s \in C^{\infty}(U, L)$ satisfies the condition of (c) for $m \in \mathbb{N}$.  
Since $\nabla''s = 0$ on $X$, 
it follows that $\nabla s (p) = (\partial \log |s|_{h}^2) s(p)$ for any $p \in X$ in the similar manner to (\ref{equation:3}).  
Then $(d \log |s|_{h}^2) s(p) = 0$ shows that $\nabla s = 0$ on $X$, 
and $(X, \nabla^{X})$ satisfies the Bohr-Sommerfeld condition.  
Hence (a), (b) and (c) are equivalent.  

Now we prove $(a) \Rightarrow (d)$.  
We take $\varphi_{0}, \xi \in C^{\infty}(U)$ as above.  
Put $\tau = \log \xi: U \to \mathbb{C}/ 2\pi \sqrt{-1} \mathbb{Z}$.  
We have $d\varphi_{0}(p) = 2d \, \mathrm{Re}\,\tau(p)$ for any $p \in X$ since $d\log |s|^2_{h}(p) = 0$.   
This implies $d^c \varphi_{0}(p) = 2d^c \,\mathrm{Re}\,\tau(p) = 2d\, \mathrm{Im}\,\tau (p)$ 
by the Cauchy-Riemann equation.  
Let $\gamma:[0, 1] \to X$ be a smooth closed curve.  
Then 
\[
\int_{\gamma} d^c \varphi_{0} = 2\int_{\gamma} d\, \mathrm{Im}\, \tau = 2\mathrm{Im}\,\tau(\gamma(1))-2\mathrm{Im}\, \tau(\gamma(0)) \in 4\pi \mathbb{Z}
\]
and (d) holds.  

Conversely, we assume (d).  
Put $\varphi_{0} = -\log |s_{0}|^2_{h}$ and 
define smooth function $g$ on X by $g = \exp(\frac{1}{2}\varphi_{0}+\frac{\sqrt{-1}}{2}\int d^{c}\varphi_{0})$.  
Using H\"ormander-Wermer Lemma, we extend $g$ to a function $\tilde{g} \in C^{\infty}(U)$ with $\overline{\partial}\tilde{g} =0$ to order $m \in \mathbb{N}$ on $X$.  
Put $\tilde{g} = e^{G}$ locally.   
%
The Cauchy-Riemann equation shows $d^c\, \mathrm{Re}\, G = d\, \mathrm{Im}\, G$ on $X$.  
Then, for any $v \in TX$, we have 
\[
\langle d (2\mathrm{Re}\,G - \varphi_{0}), v \rangle = \langle d(\varphi_{0} - \varphi_{0}), v \rangle = 0, 
\]
\[
\langle d^c (2\mathrm{Re}\, G - \varphi_{0}), v\rangle = \langle2 d \,\mathrm{Im}\, G - d^c \varphi_{0}, v\rangle = \langle d \int d^c \varphi_{0} -d^c \varphi_{0}, v \rangle = 0.  
\]
Hence $d(2\mathrm{Re}\, G -\varphi_0)(p) = 0$ for any $p \in X$.  
Define $s = \tilde{g}s_{0} \in C^{\infty}(U, L)$. 
Then $s$ satisfies $\nabla''s = 0$ to order $m$ and 
$d \log |s|_{h}^2 = 0$ on $X$.  
By multiplying $s$ by a constant, we may assume that $\log |s|_{h}^2$ is identically zero on $X$, 
and $s$ satisfies the condition of (c).  
\end{proof}
\begin{remark}
Assume there exists a non-vanishing holomorphic section $s_{l, 0} \in H^{0}(U, L^l)$.  
Put $\varphi_{0} = -\frac{1}{l}\log |s_{l, 0}|^2_{h^l}$.  
For any $\varepsilon > 0$ and any $q \in \mathbb{N}$, there exists $\varphi_{0}' \in C^{\infty}(U)$ and $l' \in \mathbb{N}$ 
such that $|\varphi_{0}-\varphi_{0}'|_{C^{q}} < \varepsilon$ on $U$, $\varphi_{0} = \varphi_{0}'$ on $X$, $\iota^{*}dd^{c}\varphi'_{0} = 0$ and that $l'\int_{\gamma} d^c \varphi_{0}' \in 4\pi \mathbb{Z}$ for any $\gamma \in H_1(X, \mathbb{Z})$ (see the proof of Lemma~3.2 of \cite{Duv-Sib}).  
Let $h' = he^{\varphi_{0}-\varphi_{0}'}$ be a Hermitian metric of $L|_{U}$. 
Then the Chern connection 
associated with $(L|_{U}^{ll'}, h'^{ll'})$ satisfies the Bohr-Sommerfeld condition on $X$.  
Furthermore, the $C^q$-norm of $|\log h'/h|$ is smaller than $\varepsilon$.    
\end{remark}
\section{Reduction to the real-analytic case}\label{section:4}
Let $M$ be a complex manifold of dimension $n$.  
Let $X \subset M$ be a compact Lagrangian submanifold of $(M, \omega)$ such that $(X, \nabla^{X})$ satisfies the Bohr-Sommerfeld condition.  
Let $U \subset M$ be a sufficiently small Stein tubular neighborhood of $X$.  
We take a non-vanishing section $s_{0} \in H^{0}(U, L)$ and put $\varphi_{0} = -\log |s_{0}|^2_{h}$.  
Take $m \in \mathbb{N}$ to be sufficiently large 
and take $s \in C^{\infty}(U, L)$ which satisfies the condition of (c) in Proposition~\ref{proposition:1}, that is, one-jet of $\log |s|_{h}^2$ vanishes on $X$ and $\nabla'' s = 0$ to order $m$ on X.  
We write $s = \xi s_{0}$ for $\xi \in C^{\infty}(U)$.  
Put $\varphi = -\log |s|_{h}^2$.  
Let $f_{k} \in H^{0}_{(2)}(M, L^k)$ which satisfies $\|f_{k}\|_{h^k} = 1$ and let $u_{k} \in C^{\infty}(U)$ such that $f_{k} = u_{k}s^k$ on $U$.  
By Whitney's theorem (Theorem~1 of \cite{Whi2}), there exists a real-analytic manifold $Y$ which is diffeomorphic to  $X$. 
By a theorem of Bruhat and Whitney (\cite{Whi-Bru}), there exists a complex manifold $N$ of complex dimension $n$, which contains $Y$ as a real analytic and totally real submanifold.  
It is possible to take $N$ as a Stein manifold (see Proposition~7 of \cite{Gra}).  
By shrinking $N$ and $U$ if necessary, 
there exists a diffeomorphism $\Phi: N \to U$ such that $\Phi(Y) = X$ and that the $m$-jet of $\overline{\partial}\Phi$ vanishes on $Y$ (cf.~Proposition~5.55 of \cite{Cie-Eli}).  
Put $\frac{1}{4\pi}dd^c (\varphi \circ \Phi) = \omega'$ and $\omega'_{n} = (\omega')^n/n!$. 
We have $\omega' = \frac{1}{4\pi}\Phi^{*} dd^c \varphi = \Phi^{*}\omega$ on $Y$ 
and $Y$ is a Lagrangian submanifold of $(N, \omega')$.   
We note that $\varphi \circ \Phi$ is a strictly plurisubharmonic function near $Y$.  
Let $d_{Y}(z)$ be the distance from $z \in N$ to $Y$ with respect to the Riemannian metric induced by $\omega'$.   
Put $W_{k} = \{z \in N\, |\, d_{Y}(z) < \frac{2\log k}{\sqrt{k}}\}$.  
\begin{lemma}\label{lemma:4}
\[
\int_{W_{k}} |\overline{\partial}(u_{k}\circ\Phi)|^2_{\omega'} e^{-k\varphi \circ \Phi}\omega'_{n} = O(k^{-m+4}).  
\]
\end{lemma}
\begin{proof}
Assume that $k \in \mathbb{N}$ is sufficiently large number.  
For simplicity we abuse notation and denote the norms of forms $|\cdot |_{\omega}$ and $|\cdot|_{\omega'}$ by $|\cdot|$.  
For example, $|(\partial u_k) \circ \Phi|$ defines a function on $N$ that assigns to each $z \in N$ a value of $|\partial u_k (\Phi(z))|_{\omega}$.  
%
We consider $\partial \Phi$ (resp., $\overline{\partial}\Phi$) as $\Phi^{*}TU$-valued one-form, and denote 
the norm of $\partial \Phi$ (resp., $\overline{\partial} \Phi$) with respect to $\omega$ and $\omega'$ by $|\partial \Phi|$ (resp., $|\overline{\partial} \Phi|$).  
Because $\overline{\partial}(u_{k}\xi^k) = 0$, we have $\overline{\partial}u_{k} = -ku_{k}\xi^{-1}\overline{\partial} \xi$.  
We have 
\begin{align*}
& |\overline{\partial} (u_{k} \circ \Phi)|^2 
\lesssim | (\partial u_{k}) \circ \Phi|^2 |\overline{\partial}\Phi|^2 + |(\overline{\partial} u_{k}) \circ \Phi|^2  |\overline{\partial} \overline{\Phi}|^2 
\lesssim | (\partial u_{k})\circ \Phi|^2 |\overline{\partial}\Phi|^2 + k^2 |(u_{k}\overline{\partial}\xi)\circ \Phi|^2 |\partial \Phi|^2.  
\end{align*}
Since $|\overline{\partial} \xi|^2= O(\frac{(\log k)^{2m}}{k^{m}})$ on $\Phi(W_{k})$ and $\int_{\Phi(W_k)}|u_k|^2e^{-k\varphi} \omega_n \leq \|f_{k}\|^2_{h^k} = 1$, we have 
\[
k^2 \int_{W_{k}} |(u_{k} \overline{\partial}\xi)\circ \Phi|^2 |\partial \Phi|^2 e^{-k\varphi \circ \Phi} \omega'_{n} 
\lesssim 
k^2 \int_{\Phi(W_{k})} |u_{k}|^2 |\overline{\partial}\xi|^2 e^{-k\varphi} \omega_{n} 
= 
O(\frac{ (\log k)^{2m}}{k^{m-2}}).  
\]
Let $\chi \in C^{\infty}(\mathbb{R})$ be a function such that $0 \leq \chi \leq 1$, $\chi = 1$ on $(-\infty, 1/2]$ and that $\chi = 0$ on $[1, +\infty)$.  
Put $\tilde{\chi}_{k} = \chi (\frac{\sqrt{k}d_{Y}}{4\log k})$.  
Then 
\[
\int_{W_{k}} |(\partial u_{k})\circ \Phi|^2 |\overline{\partial} \Phi|^2 e^{-k\varphi \circ \Phi} \omega'_{n} \leq \int_{N}\tilde{\chi}^2_{k} |(\partial u_{k})\circ \Phi|^2 |\overline{\partial} \Phi|^2 e^{-k\varphi \circ \Phi} \omega'_{n}.  
\]
Denote by $I_{k}$ the right hand side of the above inequality.  
We have 
\begin{align*}
I_{k} & \lesssim \int_{U} \tilde{\chi}^2_{k}\circ \Phi^{-1} |(\overline{\partial} \Phi)\circ \Phi^{-1}|^2 e^{-k \varphi} \partial u_{k} \wedge \overline{\partial u_{k}} \wedge \omega^{n-1} \\
& \lesssim \int_{U} \tilde{\chi}^2_{k}\circ \Phi^{-1} |(\overline{\partial} \Phi)\circ \Phi^{-1}|^2 e^{-k \varphi} |u_{k}| |\partial \overline{\partial} u_{k}| \omega_{n} 
 +  k  \int_{U}  \tilde{\chi}^2_{k}\circ \Phi^{-1} |(\overline{\partial} \Phi)\circ \Phi^{-1}|^2 e^{-k \varphi} |u_{k}| |\partial u_{k}| \omega_{n}   \\ 
& \qquad
+ \int_{U} |\partial \left((\tilde{\chi}^2_{k} |\overline{\partial} \Phi|^2) \circ \Phi^{-1} \right)| e^{-k \varphi}|u_{k}| |\partial u_{k}| \omega_{n}.  
\end{align*}
Since $|\partial \overline{\partial} u_k| =|-k \xi^{-1}\partial u_k \wedge \overline{\partial}\xi + k u_k \xi^{-2}\partial \xi \wedge \overline{\partial} \xi -ku_k \xi^{-1} \partial \overline{\partial} \xi| \lesssim k|u_k| + k |\partial u_k|$ near $X$, we have  
\begin{align*}
I_k & \lesssim 
k \left( \int_{U} \tilde{\chi}_{k}^2 \circ \Phi^{-1} |(\overline{\partial} \Phi) \circ \Phi^{-1}|^2 e^{-k \varphi} |{u}_{k}|^2 \omega_{n} 
+  \int_{U} \tilde{\chi}^2_{k}\circ \Phi^{-1} |(\overline{\partial} \Phi)\circ \Phi^{-1}|^2 e^{-k \varphi} |u_{k}| |\partial u_{k}| \omega_{n} \right.\\
& \qquad 
+\left.\int_{U} \tilde{\chi}_{k}\circ \Phi^{-1} |(\overline{\partial} \Phi)\circ \Phi^{-1}|^2 e^{-k \varphi} |u_{k}| |\partial u_{k}| \omega_{n}\right)\\
&\qquad + \int_{U} \tilde{\chi}^2_{k}\circ \Phi^{-1} |(\overline{\partial} \Phi)\circ\Phi^{-1}| |(D \overline{\partial} \Phi) \circ \Phi^{-1}|e^{-k \varphi} |u_{k}| |\partial u_{k}| \omega_{n}
\end{align*}
where $D$ is a differential operator of order one on $N$.  
We have 
\[
k \int_{U} \tilde{\chi}^2_{k}\circ \Phi^{-1} |(\overline{\partial} \Phi)\circ \Phi^{-1}|^2 e^{-k \varphi} |{u}_{k}|^2 \omega_{n}
= O(\frac{(\log k)^{2m}}{k^{m-1}}) 
\]
since $|\overline{\partial} \Phi| = O(\frac{(\log k)^m}{k^{m/2}})$ on the support of $\chi_k$ and $\int_{U}|u_k|^2 e^{-k\varphi} \omega_n \leq 1$.  
Put $W'_{k} = \{w \in N\, |\, d_{Y}(w) < \frac{4\log k}{\sqrt{k}}\}$.  
The Cauchy-Schwarz inequality implies 
\begin{align*}
& k\int_{U} \tilde{\chi}^2_{k}\circ\Phi^{-1} |(\overline{\partial} \Phi)\circ \Phi^{-1}|^2 e^{-k \varphi} |u_{k}| |\partial u_{k}| \omega_{n} 
\leq k \int_{U} \tilde{\chi}_{k}\circ\Phi^{-1} |(\overline{\partial} \Phi)\circ \Phi^{-1}|^2 e^{-k \varphi} |u_{k}| |\partial u_{k}| \omega_{n}  \\
\leq &
k \left( \int_{\Phi(W'_{k})}|(\overline{\partial} \Phi)\circ \Phi^{-1}|^2 |u_{k}|^2 e^{-k \varphi} \omega_{n} \right)^{1/2} \left(\int_{U} \tilde{\chi}^2_{k}\circ \Phi^{-1} |\partial u_{k}|^2 |(\overline{\partial} \Phi)\circ \Phi^{-1}|^2 e^{-k \varphi} \omega_{n} \right)^{1/2} \\
\lesssim & 
\frac{(\log k)^{m}}{k^{m/2-1}}
\left(\int_{N} \tilde{\chi}^2_{k} |(\partial u_{k})\circ \Phi|^2 |\overline{\partial} \Phi|^2 e^{-k \varphi\circ \Phi} \omega'_{n} \right)^{1/2} 
\lesssim 
\frac{(\log k)^{m}}{k^{m/2-1}} I_{k}^{1/2}.  
\end{align*}
Here we used $|(\overline{\partial} \Phi) \circ \Phi^{-1}| = O(\frac{(\log k)^m}{k^{m/2}})$ on $\Phi(W_k')$ and $\int_{\Phi(W_k')} |u_k|^2e^{-k\varphi} \omega_n \leq 1$.  
We also have 
\begin{align*}
& \int_{U} \tilde{\chi}^2_{k}\circ \Phi^{-1} |(\overline{\partial} \Phi)\circ \Phi^{-1}| |(D \overline{\partial} \Phi) \circ \Phi^{-1}| |u_{k}| |\partial u_{k}|e^{-k \varphi}\omega_{n} \\
\leq & 
\left(\int \tilde{\chi}^2_{k}\circ \Phi^{-1} |(D\overline{\partial} \Phi) \circ \Phi^{-1}|^2 |u_{k}|^2 e^{-k\varphi}\omega_{n} \right)^{1/2}
\left(\int \tilde{\chi}^2_{k} \circ\Phi^{-1} |(\overline{\partial} \Phi) \circ \Phi^{-1}|^2 |\partial u_{k}|^2 e^{-k\varphi}\omega_{n} \right)^{1/2} \\
\lesssim & 
\frac{(\log k)^{m-1}}{k^{m/2-1/2}} I_{k}^{1/2}.
\end{align*}
Here we used $|(D\overline{\partial} \Phi) \circ \Phi^{-1}| = O(\frac{(\log k)^{m-1}}{k^{(m-1)/2}})$ on the support of $\chi_k$.  
Hence there exists $c > 0$ which does not depend on $k$ and 
\[
I_{k} \leq c \left( \frac{1}{k^{m-2}} + \frac{1}{k^{m/2-2}} I_{k}^{1/2} \right) 
\leq 
c \left(\frac{1}{k^{m-2}} + \frac{2c}{k^{m-4}} + \frac{1}{2c}I_{k} \right).  
\]
We have $I_{k} = O(k^{-m+4})$ and we complete the proof.  
\end{proof}
Now we introduce the H\"ormander's $L^2$-estimate for $\overline{\partial}$-equation.  
\begin{theorem}{$($cf.~Corollary~5.3 of \cite{Dem2}}\label{theorem:4}$)$
Let $M$ be a Stein manifold of dimension $n$ with a K\"ahler metric $\omega$.   
Let $L$ be a holomorphic line bundle over $M$.  
Suppose that there exists a Hermitian metric $h$ (which may be singular) of $L$ such that its Chern form $c_{1}(L, h)$ 
satisfies 
\begin{equation}\label{inequality:1}
c_{1}(L, h) \geq C \omega
\end{equation}
on $M$ for some $C>0$.  
Then, for any $L$-valued square integrable $(n, 1)$-form $g$ such that $\overline{\partial} g = 0$ on $M$, there exists an $L$-valued $(n, 0)$-form $f$ which satisfies $\overline{\partial} f = g$ and 
\[
\int_{M} |f|^2_{h, \omega} \omega_{n} \leq  C' \int_{M} |g|^2_{h, \omega} \omega_n.  
\]
%
Here $C' > 0$ is a positive real number which depends only on $C$.  
\end{theorem}
For sufficiently large $k \in \mathbb{N}$, $W_{k}$ is a Stein manifold.  
By regarding $u_{k} \circ \Phi$ as a $\bigwedge^n T^{(1, 0)}N|_{W_{k}}$-valued $(n, 0)$-form, 
Theorem~\ref{theorem:4} implies that 
there exists $v_{k} \in C^{\infty}(W_{k})$ such that $\overline{\partial}v_{k} = \overline{\partial} (u_{k}\circ \Phi)$ and that $\int_{W_{k}} |v_{k}|^2 e^{-k \varphi \circ \Phi}\omega'_{n} = O(k^{-m+4})$.  
Let $\beta_{k} = u_{k}\circ \Phi -v_{k}$.  
Then $\beta_{k}$ is a holomorphic function on $W_{k}$.  
For any $\varepsilon > 0$, 
we have 
\begin{align*}
& \int_{W_k}|\beta_k|^2 e^{-k\varphi \circ \Phi} \omega'_n \leq (1 + \varepsilon) \int_{W_k}|u_k \circ \Phi|^2e^{-k\varphi \circ \Phi}\omega_n' + \left(1 + \frac{1}{\varepsilon}\right)\int_{W_k} |v_k|^2 e^{-k\varphi \circ \Phi} \omega_n' \\
\leq & 1 + \varepsilon + \left(1 + \frac{1}{\varepsilon}\right)\int_{W_k} |v_k|^2 e^{-k\varphi \circ \Phi} \omega_n'.   
\end{align*}
Hence, if $m > 4$, we have 
\begin{equation}\label{equation:5} 
\limsup_{k \to +\infty}\int_{W_{k}}|\beta_{k}|^2 e^{-k\varphi \circ \Phi} \omega'_{n}  \leq 1.   
\end{equation}
Our next task is to estimate $|v_{k}|$ on $Y$.  
We first estimate $u_{k}$ near $X$.  
\begin{lemma}\label{lemma:5}
Let $r > 0$ be a sufficiently small number.   
Let $c > 0$ be a positive number which is larger than the $C^1$-norm of $\varphi_{0}$ on a neighborhood of $X$.  
Let $D$ be a differential operator of order one with bounded coefficients.  
Then there exists $\delta > 0$ which does not depend on $r$ and the following estimates hold:  
\begin{itemize}
\item[(i)]  
$\sup_{\{z \in M, d_{X}(z) < \delta r\}} |u_{k}|^2 e^{-k\varphi} \lesssim e^{ckr}/r^{2n}, $

\item[(ii)]  
$\sup_{\{z \in M, d_{X}(z) < \delta r\}} |Du_{k}|^2 e^{-k\varphi} \lesssim e^{4ckr}/r^{2+2n} + 
k^2e^{ckr}/r^{2n}.  $

\end{itemize}
\end{lemma}
\begin{proof}
We take open sets $\{V_{i}\}_{i=1}^l$, $\{V'_{i}\}_{i=1}^l$ ($V_{i} \subset  \subset V'_{i}$) in $M$ such that $X \subset \bigcup_{i=1}^l V_{i}$ and that there exists a holomorphic coordinate on $V'_{i}$. 
Let $B_{i}(p, r) \subset V'_{i}$ be the Euclidean ball of center $p \in V'_{i}$ and radius $r$ in $V'_{i}$.  
We may assume that $B_{i}(p, 2r) \subset V'_{i}$ for any $p \in X \cap V_{i}$ by taking $r$ sufficiently small.  
Let $p \in X \cap V_{i}$.  
Since $u_{k}\xi^k$ is a holomorphic function, the mean value inequality implies 
\[
\sup_{B_{i}(p, r)} |u_{k}\xi^k|^2e^{-k\varphi_{0}} \lesssim \frac{e^{ckr}}{r^{2n}} \int_{B_{i}(p, 2r)} |u_{k}\xi^k|^2 e^{-k\varphi_{0}}d\mu_{i, \mathrm{Leb}} \lesssim \frac{e^{ckr}}{r^{2n}} 
\]
where $d\mu_{i, \mathrm{Leb}}$ is the Lebesgue measure on $V'_{i}$.  
Since 
$
\sup_{B_{i}(p, r)}|D(u_{k}\xi^k)| \lesssim \frac{1}{r} \sup_{B_{i}(p, 2r)} |u_{k}\xi^k|, 
$
we have 
\[
\sup_{B_{i}(p, r)}|D(u_{k} \xi^k)|^2 e^{-k\varphi_{0}} \lesssim  \frac{e^{2ckr}}{r^2} \sup_{B_{i}(p, 2r)}|u_{k}\xi^k|^2 e^{-k\varphi_{0}}.  
\]
Hence 
\[
\sup_{B_{i}(p, r)} |D u_{k}|^2 e^{-k\varphi} \lesssim \frac{e^{2ckr}}{r^{2}} \sup_{B_{i}(p, 2r)}|u_{k}|^2e^{-k\varphi}  + k^2 \sup_{B_{i}(p, r)}  |u_{k}|^2e^{-k\varphi} 
\lesssim \frac{e^{4ckr}}{r^{2+2n}} + \frac{k^2e^{ckr}}{r^{2n}}.  
\]
If we take $\delta$ sufficiently small, $\bigcup_{i=1}^l \bigcup_{p \in X \cap V_{i}} B_{i}(p, r)$ contains $\{z \in M\,|\, d_{X}(z) < \delta r \}$ for any sufficiently small $r$.  
This completes the proof of (i), (ii).  
\end{proof}

\begin{lemma}\label{lemma:6}
We have 
\[
|v_{k}|^2 = O(k^{2n-2m}) + O(k^{2n-m+4})
\]
uniformly on $Y$.  
\end{lemma}
\begin{proof}
Let $q \in Y$ and $W_q \subset N$ be a sufficiently small neighborhood of $q$ and $(w_{1}, \ldots, w_{n})$ be a holomorphic coordinate on $W_q$.  
Let $\delta > 0$ be as in Lemma~\ref{lemma:5}.  
Then there exists $\delta' > 0$ which does not depend on $q \in Y$ such that $\Phi (B(q, \delta' r)) \subset \{z \in M\, |\, d_{X}(z) < \delta r\}$ for any $0 < r << 1$.  
Here $B(q, \delta'r)$ is the Euclidean ball of center $q$ and radius $\delta' r$ in $W_q$.  
We take $r$ to be a positive number which depends only on $k$ such that $B(q, \delta'r) \subset W_{k} = \{z \in N \,|\, d_{Y}(z) < \frac{2\log k}{\sqrt{k}}\}$.  
Let $c' > 0$ be a positive number which is larger than the $C^1$-norm of $\varphi \circ \Phi$ on $W_q$.  
By Lemma~15.1.8 of \cite{Hor}, we have 
\begin{align*}
|v_{k}(q)|^2 &\lesssim (\delta' r)^2 \sup_{B(q, \delta'r)} |\overline{\partial}_{w} v_{k}|^2 + \frac{1}{(\delta'r)^{2n}} \int_{B(q, \delta'r)} |v_{k}|^2 d\mu_{\mathrm{Leb}} \\
& \lesssim r^2 e^{c'\delta'kr} \sup_{B(q, \delta' r)} |\overline{\partial}_{w}v_{k}|^2e^{-k\varphi \circ \Phi} + r^{-2n}e^{c'\delta'kr} k^{-m+4}.  
\end{align*}
Let $(z_{1}, \ldots, z_{n})$ be a holomorphic coordinate on a neighborhood of $\Phi (q)$.  
Then 
\begin{align*}
|\overline{\partial}_{w}v_{k}| = |\overline{\partial}_{w}(u_{k} \circ \Phi)| \leq |\partial_{z} u_{k}| |\overline{\partial}_{w} \Phi| + |\overline{\partial}_{z} u_{k}| |\partial_{w} \Phi| 
\lesssim 
|\partial_{z}u_{k}||\overline{\partial}_{w} \Phi| + k\big|u_{k}\overline{\partial}_{z}\xi\big|.  
\end{align*}
By Lemma~\ref{lemma:5}, we have 
\begin{align*}
& |v_{k}(q)|^2 \\
 \lesssim & r^{-2n} e^{c'\delta' kr} ((e^{4ckr} + r^2k^2 e^{ckr} )\sup_{B(q, \delta' r)}|\overline{\partial}_{w} \Phi|^2 + r^2k^2 e^{ckr}\sup_{B(q, \delta' r)}|\overline{\partial}_{z}\xi|^2 ) + r^{-2n}e^{c'\delta'kr}k^{-m+4} \\
\lesssim &
r^{-2n + 2m}e^{c'\delta' kr}  \left(e^{4ckr} + r^2k^2e^{ckr} \right) + r^{-2n} e^{c' \delta' kr} k^{-m+4}.  
\end{align*}
If we take $r = k^{-1}$, we have $B(q, \delta'r) \subset W_{k}$ for sufficiently large $k$ and 
$|v_{k}(q)|^2 = O(k^{2n-2m}) + O(k^{2n-m+4})$.  
Since $Y$ is compact, the above estimates hold uniformly on $Y$.  
\end{proof}

\section{Monge-Amp\`ere equation and Demailly's Jensen-Lelong formula}\label{section:5}
We use the same notation as in Section~\ref{section:4}.  
Since $N$ is Stein, we may assume that $N$ is a real-analytic submanifold in a higher dimensional Euclidean space.  
Then one can approximate $\varphi \circ \Phi$ by analytic functions in the $C^l$-norm for any $l \in \mathbb{N}$ on a neighborhood of $Y$ (see Lemma~5 of \cite{Whi}).  
Let $\varepsilon > 0$ be a sufficiently small number.  
Let $\varphi_{\varepsilon}$ be a real analytic function on a neighborhood of $Y$ such that 
\begin{equation}\label{equation:7}
|\varphi_{\varepsilon}-\varphi\circ\Phi|_{C^2} < \varepsilon. 
\end{equation} 
Let $ds^2$ be the real-analytic Riemannian metric on $Y$ induced by $\frac{1}{4\pi}dd^c \varphi_{\varepsilon}$.  
By the fundamental result of Guillemin and Stenzel~\cite{Gui-Ste}, 
there exists the real-analytic strictly plurisubharmonic function $\rho$ on a neighborhood of $Y$ such that $0 \leq \rho \leq 1$, $\rho^{-1}(0) = Y$, the K\"ahler form $\frac{1}{4\pi}dd^c \rho$ induces the Riemannian metric $ds^2$ on $Y$ and that 
$
(dd^c \sqrt{\rho})^n = 0
$
outside $Y$.  
We note that $\rho$ depends on $\varepsilon > 0$.  
By Shrinking $N$ if necessary, we may assume that $\rho$ is defined on $N$.  
Since $\rho$ is strictly plurisubharmonic and $\sqrt{\rho}$ is a solution of the Monge-Amp\`ere equation, $\partial \overline{\partial}\sqrt{\rho}$ has $(n-1)$ positive eigenvalues and does not have negative eigenvalues on $N \setminus Y$.  
Because $\sqrt{\rho}$ attains its minimum at $Y$, 
we have that $\sqrt{\rho}$ is plurisubharmonic on $N$ (cf.~Theorem~5.6 of \cite{Lem-Szo}).  
Although $\sqrt{\rho}$ is continuous and not of class $C^2$ on $N$, the Monge-Amp\`ere measure $(dd^c \sqrt{\rho})^n$ can still be defined.  
See Chapter~III, Section~3 of \cite{Dem} for the definition of the Monge-Amp\`ere measure.  
Since $(dd^c \sqrt{\rho})^n = 0$ on $N \setminus Y$, the support of the measure $(dd^c \sqrt{\rho})^n$ is contained in $Y$.  

Put $\omega_{\rho} = \frac{1}{4\pi}dd^c \rho$.  
Let $q \in Y$ and $W_{q} \subset N$ be a small open neighborhood of $q$.  
Let $x = (x_{1}, \ldots, x_{n})$ be a smooth coordinate on $W_{q} \cap Y$ and $\{e_{1}, \ldots, e_{n}\}$ be an orthonormal frame of $TY$ on $W_{q} \cap Y$ with respect to the Riemannian metric induced by $\omega_{\rho}$.  
Let $J_{N} \in \mathrm{End} (TN)$ be the complex structure of $N$.  
Then $(J_{N}e_{1}, \ldots, J_{N}e_{n})$ induces a frame of the normal bundle $TN/TY$ on $W_{q} \cap Y$.  
We have $\langle e_{1}, \ldots, e_{n}\rangle^{\bot} = \langle J_{N}e_{1}, \ldots, J_{N}e_{n}\rangle$.  
Let $(x, y) = (x_{1}, \ldots, x_{n}, y_{1}, \ldots, y_{n})$ be the coordinate on $W_q$ associated to $(J_{N}e_{1}, \ldots, J_{N}e_{n})$, 
that is, $(x_{1}, \ldots, x_{n}, 0, \ldots, 0)$ corresponds to a point in $W_{q} \cap Y$ and 
$(x_{1}, \ldots, x_{n}, y_{1}, \ldots, y_{n})$ corresponds to $\exp_{(x, 0)} (y_{1}J_{N}e_{1} + \cdots + y_{n}J_{N}e_{n}) \in N$.  
Then 
$
\rho (x, y) = 2\pi \sum_{i=1}^n y_{i}^2 + O(|y|^3)
$
on $W_q$ 
since $\rho$ attains its minimum at $Y$, and $2\pi dd^c \rho (\partial/\partial y_i, J_N \partial /\partial y_j)(x, 0) = \delta_{ij}$.  
On the other hand, we have 
\[
\varphi \circ \Phi (x, y) = \sum_{i=1}^{n}\sum_{j=1}^n a_{ij}(x) y_{i}y_{j} + O(|y|^3)
\]
where $a_{ij}(x) = \frac{1}{2}dd^c (\rho\circ \Phi) (\partial/\partial y_i, J_N\partial/\partial y_j) =2\pi \omega' (\partial/\partial y_{i}, J_N \partial/\partial y_{j})(x, 0)$.  
Here $\omega' = \frac{1}{4\pi}dd^c (\varphi \circ \Phi)$. 
By (\ref{equation:7}), we have 
$
|a_{ij} - 2\pi dd^c \varphi_{\varepsilon}(\partial/\partial y_{i}, J_N \partial/\partial y_{j}) | \lesssim \varepsilon, 
$
and the left-hand side is equal to 
\begin{align*}
& |a_{ij} - \frac{1}{2}ds^2(J_N\partial/\partial y_{i}, J_N \partial/\partial y_{j})| = |a_{ij}- 2\pi \omega_{\rho}(\partial/\partial y_{i}, J_N \partial/\partial y_{j})| 
= |a_{ij} -  2\pi \delta_{ij}|.  
\end{align*}
Then there exist $c_{1}, c_{2}, c_{3} > 0$ which do not depend on $k$ or $\varepsilon$ such that 
\begin{equation}\label{equation:8}
(1-c_{1}\varepsilon) \varphi \circ \Phi \leq \rho \leq (1+c_{1}\varepsilon) \varphi \circ \Phi,  
\end{equation}
\begin{equation}\label{equation:9}
(1-c_{2}\varepsilon) \omega'_{n} \leq \omega_{\rho, n} \leq (1+c_{2}\varepsilon) \omega'_{n},
\end{equation}
\begin{equation}\label{equation:10}
(1-c_{3}\varepsilon) \mathrm{Vol}(X, \omega) \leq \mathrm{Vol}(Y, \omega_{\rho}) \leq (1+c_{3}\varepsilon) \mathrm{Vol}(X, \omega)
\end{equation}
on a neighborhood of $Y$.   
Here $\omega_{\rho, n} = \omega_{\rho}^n/n!$ and $\mathrm{Vol}(Y, \omega_{\rho})$ is defined by the integral of the Riemannian density induced by $\omega_{\rho}$ on $Y$.  
We define $B_{\sqrt{\rho}}(r) = \{z \in N\, |\, \sqrt{\rho}(z) \leq r \}$, $S_{\sqrt{\rho}}(r) = \{z \in N\, |\, \sqrt{\rho}(z) = r \}$.  
For sufficiently large $k \in \mathbb{N}$, we have $B_{\sqrt{\rho}}(\frac{\sqrt{2\pi}\log k}{\sqrt{k}}) \subset W_{k} = \{z \in N\,|\, d_{Y}(z) < \frac{2\log k}{\sqrt{k}}\}$.  
Now we introduce Demailly's Jensen-Lelong formula: 
\begin{theorem}[(6.5) of \cite{Dem}]\label{theorem:DJL}
Let $N$ be a Stein manifold and $\phi$ be a continuous plurisubharmonic function on $N$.  
Assume that the sublevel set $B_{\phi}(r) = \{z \in N\, |\,  \phi(z) < r\}$ is relatively compact for any $r < \sup_{N}\phi$.  
Then 
\[
\int u d\mu_{r} - \int_{B_{\phi}(r)} u (dd^c \phi)^n = \int_{-\infty}^{r} dt \int_{B_{\phi}(t)}  dd^c u \wedge (dd^c \phi)^{n-1} 
\]
for any $u \in C^{\infty}(N)$.  
Here $d\mu_r$ is a measure whose support is contained in $\{z \in N\, |\, \phi(z) = r\}$.  
If $\phi$ is smooth and $d\phi \neq 0$ on a neighborhood of $\partial B_{\phi}(r)$, 
$d\mu_r$  is equal to the pullback of $d^c \phi \wedge (dd^c\phi)^{n-1}$ by the inclusion map from $\partial B_{\phi}(r)$ to $N$.   
\end{theorem}
By Demailly's Jensen-Lelong formula, 
we have 
\begin{align*}
\inf_{y \in Y} |\beta_{k}(y)|^2 \int_{S_{\sqrt{\rho}}(r)} d^c \sqrt{\rho} \wedge (dd^c \sqrt{\rho})^{n-1} 
& = \inf_{y \in Y} |\beta_{k}(y)|^2 \int_{Y} (dd^c \sqrt{\rho})^n \leq \int_{Y} |\beta_{k}|^2 (dd^c \sqrt{\rho})^n \\ 
& \leq \int_{S_{\sqrt{\rho}}(r)} |\beta_{k}|^2 d^c \sqrt{\rho} \wedge (dd^c \sqrt{\rho})^{n-1}    
\end{align*}
for $0 < r < \frac{\log k}{\sqrt{k}}$.  
Here $(dd^c \sqrt{\rho})^n$ is the Monge-Amp\`ere measure whose support is contained in $Y$.  
Since $d^c \sqrt{\rho} = \rho^{-1/2}\frac{d^c \rho}{2}$ and 
$dd^c \sqrt{\rho} = \rho^{-1/2}\frac{dd^c \rho}{2}-\rho^{-3/2}\frac{d\rho \wedge d^c\rho}{4}$, 
we obtain 
\[
\inf_{y \in Y} |\beta_{k}(y)|^2 \int_{S_{\sqrt{\rho}}(r)}  \rho^{-n/2}d^{c} \rho \wedge (dd^c \rho)^{n-1} \leq \int_{S_{\sqrt{\rho}}(r)} |\beta_{k}|^2 \rho^{-n/2}d^c \rho \wedge (dd^c \rho)^{n-1} .  
\]
Since $\rho$ is constant on $S_{\sqrt{\rho}}(r)$, we have 
\[
\inf_{y \in Y} |\beta_{k}(y)|^2 \int_{S_{\sqrt{\rho}}(r)} e^{-kb\rho} \rho^{-1/2}d^{c} \rho \wedge (dd^c \rho)^{n-1} \leq \int_{S_{\sqrt{\rho}}(r)} |\beta_{k}|^2 e^{-k b\rho} \rho^{-1/2}d^c \rho \wedge (dd^c \rho)^{n-1} 
\]
for any $b > 0$.  
Thus, 
\begin{align*}
& \inf_{y \in Y} |\beta_{k}(y)|^2 \int_{B_{\sqrt{\rho}}(\frac{\sqrt{2\pi}\log k}{\sqrt{k}})} e^{-kb\rho} \rho^{-1}d\rho \wedge d^{c} \rho \wedge (dd^c \rho)^{n-1} \\
= & 2\inf_{y \in Y} |\beta_{k}(y)|^2 \int_{0}^{\frac{\sqrt{2\pi}\log k}{\sqrt{k}}} dr \int_{S_{\sqrt{\rho}}(r)} e^{-kb\rho} \rho^{-1/2} d^c \rho \wedge (dd^c \rho)^{n-1} \\
\leq & 2\int_{0}^{\frac{\sqrt{2\pi}\log k}{\sqrt{k}}} dr \int_{S_{\sqrt{\rho}}(r)} |\beta_{k}|^2 e^{-kb\rho}\rho^{-1/2} d^c \rho \wedge (dd^c \rho)^{n-1} \\
= & \int_{B_{\sqrt{\rho}}(\frac{\sqrt{2\pi}\log k}{\sqrt{k}})} |\beta_{k}|^2 e^{-kb\rho} \rho^{-1} d\rho \wedge d^c \rho \wedge (dd^{c}\rho)^{n-1}.  
\end{align*}
Since $(dd^c \sqrt{\rho})^n = 0$, 
we have 
\[
n \rho^{-1}d\rho \wedge d^c \rho \wedge (dd^c \rho)^{n-1} =  2 (dd^c \rho)^n.
\]
Hence we have 
\begin{equation}\label{equation:11}
\inf_{y \in Y} |\beta_{k}(y)|^2 \int_{B_{\sqrt{\rho}}(\frac{\sqrt{2\pi}\log k}{\sqrt{k}})} e^{-kb\rho} \omega_{\rho, n} 
\leq \int_{B_{\sqrt{\rho}}(\frac{\sqrt{2\pi}\log k}{\sqrt{k}})} |\beta_{k}|^2 e^{-kb\rho} \omega_{\rho, n}.
\end{equation}

\begin{lemma}\label{lemma:7}
\[
\int_{B_{\sqrt{\rho}}(\frac{\sqrt{2\pi}\log k}{\sqrt{k}})} e^{-kb\rho} \omega_{\rho, n} \sim \frac{\mathrm{Vol}(Y, \omega_{\rho})}{(2kb)^{n/2}} \qquad (k \to +\infty).  
\]
\end{lemma}
\begin{proof}
Let $\{q_{j}\}_{j = 1}^{l} \subset Y$ and $W_{j} \subset N$ be a small neighborhood of $q_{j}$ such that $Y \subset \bigcup_{j=1}^l W_{j}$ and that there exist 
non-negative smooth functions $\lambda_{j} \in C^{\infty}_{0}(W_{j})$ which satisfy $\sum_{j=1}^{l}\lambda_{j} = 1$ on a neighborhood of $Y$.  
We take a smooth coordinate $(x, y)$ on $W_j$ as in the first part of this section.  
Since $\rho(x, y) = 2\pi \sum_{i=1}^{n} y_{i}^2 + O(|y|^3)$, 
there exists $c_{4} > 1$ which does not depend on $k$ and satisfies 
\[
\{z \in W_j \, |\, |y(z)| < c_{4}^{-1}\frac{\log k}{\sqrt{k}} \} \subset   B_{\sqrt{\rho}}(\frac{\sqrt{2\pi}\log k}{\sqrt{k}}) \subset \{z \in W_j \, |\, |y(z)| < c_{4}\frac{\log k}{\sqrt{k}}\}
\]
for sufficiently large $k$.  
We have 
\begin{align*}
& \int_{|y| < c_{4}^{-1} \frac{\log k}{\sqrt{k}}} \lambda_{j}e^{-kb\rho}\omega_{\rho, n}  = \int_{|y| < c_{4}^{-1} \frac{\log k}{\sqrt{k}}} \lambda_{j}(x, y) e^{-2\pi b k|y|^2 + kO(|y|^3)} (1+ O(|y|)) d\mu_{Y}dy \\
= & \frac{1}{k^{n/2}}\int_{|y| < c_{4}^{-1}\log k} \lambda_{j}(x, \frac{y}{\sqrt{k}}) e^{-2\pi b |y|^2} \left(1 + O\left(\frac{(\log k)^3}{\sqrt{k}}\right) \right) d\mu_{Y}dy
\end{align*}
where $d\mu_{Y}$ is the Riemannian density on $Y$ induced by $\omega_{\rho}$.  
Since 
\[
\int_{c_{4}^{-1} \log k \leq |y| < c_4 \log k}\lambda_{j}(x, 0) e^{-2\pi b |y|^2} d\mu_{Y}dy = O(k^{-\infty}), 
\]
it follows that 
\begin{align*}
\int_{B_{\sqrt{\rho}}(\frac{\sqrt{2\pi} \log k}{\sqrt{k}})} \lambda_{j} e^{-kb\rho}\omega_{\rho, n} 
= &\frac{1}{k^{n/2}} \left(1 + O\left(\frac{(\log k)^3}{\sqrt{k}}\right) \right) \int \lambda_{j}(x, 0) e^{-2\pi b |y|^2} d\mu_{Y} dy + O(k^{-\infty}) \\
= & \frac{1}{(2kb)^{n/2}} \left(1 + O\left(\frac{(\log k)^3}{\sqrt{k}}\right) \right)  \int \lambda_{j}(x, 0) d\mu_{Y} + O(k^{-\infty}).  
\end{align*}
Hence we obtain $\int_{B_{\sqrt{\rho}}(\frac{\sqrt{2\pi} \log k}{\sqrt{k}})} e^{-kb\rho}\omega_{\rho, n} = \frac{1}{(2kb)^{n/2}} \left(1 + O\left(\frac{(\log k)^3}{\sqrt{k}}\right) \right) \mathrm{Vol}(Y, \omega_{\rho})$.  
\end{proof}
Now put $b = \frac{1}{1-c_{1}\varepsilon}$.  
By (\ref{equation:5}), (\ref{equation:8}) and (\ref{equation:9}), we have 
\[
\limsup_{k \to +\infty}\int_{B_{\sqrt{\rho}}(\frac{\sqrt{2\pi} \log k}{\sqrt{k}})} |\beta_{k}|^2 e^{-kb \rho} \omega_{\rho, n} \leq 1 + c_{2}\varepsilon.  
\]
Then (\ref{equation:10}), (\ref{equation:11}) and Lemma~\ref{lemma:7} imply 
\[
\limsup_{k \to \infty} \inf_{y \in Y} |\beta_{k}(y)|^2 \frac{\mathrm{Vol}(X, \omega)}{(2k)^{n/2}} \leq \frac{1 + c_2 \varepsilon}{(1-c_1 \varepsilon)^{n/2}(1-c_3 \varepsilon)}.  
\]
By Lemma~\ref{lemma:6}, we have 
\begin{align*}
\limsup_{k \to +\infty}\frac{\inf_{x \in X}|f_k(x)|^2_{h^k}\mathrm{Vol}(X, \omega)}{(2k)^{n/2}} \leq &\limsup_{k \to +\infty}\inf_{y \in Y} \left((1+ \varepsilon) |\beta_k(y)|^2 + \left(1 + \frac{1}{\varepsilon} \right)|v_{k}(y)|^2\right) \frac{\mathrm{Vol}(X, \omega)}{(2k)^{n/2}} \\
\leq & \frac{(1 + \varepsilon)(1+c_2 \varepsilon)}{(1-c_1 \varepsilon)^{n/2}(1-c_3\varepsilon)}
\end{align*}
for sufficiently large $m$.  
This completes the proof of Theorem~\ref{theorem:1} since $\varepsilon$ is any positive number.  

\section{Estimate from below}\label{section:3}
Let $M$ be a complex manifold of dimension $n$.  
Let $X \subset M$ be a compact Lagrangian submanifold of $(M, \omega)$ such that $(X, \nabla^{X})$ satisfies the Bohr-Sommerfeld condition.  
We take $U \subset M$, $m \in \mathbb{N}$, $s_0 \in H^{0}(U, L)$, $s \in C^{\infty}(U, L)$ and $\xi, \varphi_{0}, \varphi \in C^{\infty}(U)$ as in Section~\ref{section:4}.  
Let $p \in X$ and $V \subset M$ be a small neighborhood of $p$.  
We take a smooth coordinate $x = (x_{1}, \ldots, x_{n})$ on $V \cap X$ and take a local frame $(e_{1}, \ldots, e_{n})$ of $TX$ on $V \cap X$.  
We take a local coordinate system $(x, y) = (x_{1}, \ldots, x_{n}, y_{1}, \ldots, y_{n})$ as in the first part of Section~\ref{section:5}.  
Just as in the case of $\rho$, 
we have 
$\varphi(x, y) = 2\pi \sum_{i=1}^n y_i^2 + O(|y|^3)$.  
Let $d_{X}(z)$ be the distance from $z \in M$ to $X$.  
Let $\chi \in C^{\infty}(\mathbb{R})$ be a function such that $0 \leq \chi \leq 1$, $\chi = 1$ on $(-\infty, 1/2]$ and that $\chi = 0$ on $[1, +\infty)$.  
Define $\chi_{k}(z) = \chi(\frac{\sqrt{k}d_{X}(z)}{\log k})$ for $z \in M$.   
Put $s_{k} = \chi_{k} s^k \in C^{\infty}(U, L^k)$ for $k \in \mathbb{N}$.  
\begin{lemma}\label{lemma:1}
\[
\|\nabla'' s_{k}\|^2_{h^k, \omega} = O(k^{2-m}).  
\]
\end{lemma}
\begin{proof}
Assume that $k \in \mathbb{N}$ is sufficiently large.  
We have 
\begin{align*}
|\nabla'' s_{k}|^2_{h^k, \omega} & = |\xi^k \overline{\partial}\chi_{k} +k\chi_{k}\xi^{k-1}\overline{\partial}\xi|^2_{\omega} e^{-k\varphi_{0}} \lesssim |\overline{\partial}\chi_{k}|^2_{\omega}e^{-k\varphi} + k^2|\chi_{k}\overline{\partial}\xi|^2_{\omega}e^{-k\varphi}.  
\end{align*}
It follows that 
\begin{align*}
& \int_{V} |\overline{\partial}\chi_{k}|^2_{\omega}e^{-k\varphi} \omega_{n} \lesssim \left(\frac{\sqrt{k}}{\log k}\right)^2 \int_{\frac{\log k}{2\sqrt{k}}< |y| < \frac{\log k}{\sqrt{k}}, (x, y) \in V} e^{-2k\pi|y|^2} dxdy \\
\lesssim & \left(\frac{\log k}{\sqrt{k}}\right)^{n-2}e^{-\pi (\log k)^2/2} 
= O(k^{-\infty})  
\end{align*}
and 
\begin{align*}
k^2 \int_{V} |\chi_{k} \overline{\partial}\xi|^2_{\omega}e^{-k\varphi} \omega_{n} 
\lesssim k^2 \int_{|y|< \frac{\log k}{\sqrt{k}}, (x, y) \in V} |y|^{2m} e^{-2k\pi|y|^2}dxdy = O(k^{2-m}).  
\end{align*}
The last equality holds by the boundedness of $|\sqrt{k}y|^{2m}e^{-\pi|\sqrt{k}y|^2}$.  
The lemma holds since $X$ is compact.  
\end{proof}

Let $A$ be a finite sequence of points in $M \setminus X$ (possibly empty).  
Assume that $A$ consists of $a_1, \ldots, a_N \in M \setminus X$ ($a_i \neq a_j$ if $i \neq j$) and $a_i$ occurs $l_i$ times in $A$.  
Let $U_{i} \subset M$ be a small neighborhood of $a_i$.  
We may assume that the support of $\chi_{k}$ does not intersect $U_{i}$ for any $i$ and $k$.  
Let $0 \leq \kappa_i \leq 1$ be a smooth function on $M$ such that the support of $\kappa_i$ is contained in $U_i$ and $\kappa_i = 1$ on a neighborhood of $a_i$.  
Let $z = (z_1, \ldots, z_n)$ be a holomorphic coordinate on $U_i$ such that $a_i$ corresponds to the origin. 
Then we put $\tau_{i}(z) = \kappa_i (z) (n+l_i-1)\log |z|^2$.  
By taking $U_i$ small, we have $\tau_{i} \leq 0$.  
We define a singular Hermitian metric $h'_{k}$ of $L^k$ by $h'_{k} = h^k e^{-\sum_{i=1}^{N} \tau_{i}}$.  
%
Observe that if a holomorphic section $g$ of $L^k$ satisfies $\|g\|_{h'_{k}} < \infty$, then $g$ vanishes to order $l_i$ at $a_i$.   
The Chern form induced by $h'_{k}$ is positive in the sense of currents if $k$ is sufficiently large.  
Lemma~\ref{lemma:1} shows that $\| \nabla'' s_{k}\|^2_{h'_{k}, \omega} = O(k^{2-m})$.  
(Note that $h^k = h'_{k}$ on the support of $s_k$.)

Now we assume that $M$ satisfies one of the three conditions in Theorem~\ref{theorem:3}.  
We regard $s_{k}$ as an $L^k \otimes \bigwedge^{n} T^{(1, 0)}M$-valued $(n, 0)$-form.  
By using Theorem~\ref{theorem:4}, 
we show that there exists $t_{k} \in C^{\infty}(M, L^k)$ such that $\nabla'' t_{k} = \nabla'' s_{k}$ and $\|t_k\|^2_{h^k} \leq \|t_{k}\|^2_{h'_k} \lesssim \|\nabla'' s_{k}\|^2_{h'_k, \omega}$ for sufficiently large $k$.   
If $M$ is a pseudoconvex domain in $\mathbb{C}^n$, $\bigwedge^n T^{(1, 0)}$ is a trivial line bundle and we can use Theorem~\ref{theorem:4} directly.  
If $M$ is a projective (resp.~Stein) manifold, we need to deal with the difference of the Chern form of $L^k$ and $L^k \otimes \bigwedge^n T^{(1, 0)} M$ to verify whether the condition (\ref{inequality:1}) holds.  
However the compactness (resp.~ the condition $Ric (\omega) \geq -C\omega$) implies (\ref{inequality:1}) and gives the solution of $\overline{\partial}$-equation with the desired estimate for sufficiently large $k$.  
We define a holomorphic section $\alpha_{k} = s_{k} - t_{k}$.  
%
Since $\|\alpha_k\|^2_{h'_{k}} \leq 2(\|s_k\|^2_{h'_k} + \|t_k\|^2_{h'_k}) < + \infty$, $\alpha_k$ vanishes to order $l_{i}$ at $a_i$ and 
$\alpha_k \in H^{0}_{(2), A} (M, L^k)$.   
\begin{lemma}\label{lemma:2}
If $m > n/2 + 2$, we have 
\[
\|\alpha_{k}\|^2_{h^k} \sim \frac{\mathrm{Vol}(X, \omega)}{(2k)^{n/2}} \quad (k \to +\infty).  
\]
\end{lemma}
\begin{proof}
For any $\varepsilon > 0$, 
it follows that 
\[
(1-\varepsilon) \|s_{k}\|^2_{h^k} + \left(1-\frac{1}{\varepsilon}\right) \|t_{k}\|_{h^k}^2 \leq \|\alpha_{k}\|^2_{h^k} \leq (1+\varepsilon) \|s_{k}\|^2_{h^k} + \left(1+\frac{1}{\varepsilon}\right) \|t_{k}\|^2_{h^k}.  
\]
If $m > n/2 + 2$, we have $\lim_{k \to \infty}\frac{\|t_k\|^2_{h^k} (2k)^{n/2}}{\mathrm{Vol} (X, \omega)} = 0$ since $\|t_k\|_{h^k}^2 = O(k^{2-m})$.  
Hence it is enough to show $\|s_{k}\|^2_{h^k} \sim \frac{\mathrm{Vol}(X, \omega)}{(2k)^{n/2}}$ ($k \to +\infty$) and we can prove Lemma~\ref{lemma:2} by the same argument as in Lemma~\ref{lemma:7}.  
\end{proof}

\begin{lemma}\label{lemma:3}
We have 
\[
|t_{k}|^2_{h^k} = O(k^{2+2n-m})
\]
uniformly on $X$.  
\end{lemma}

\begin{proof}
Let $k \in \mathbb{N}$ be a sufficiently large number.  
Let $p \in X$ and 
$(V, (z_{1}, \ldots, z_{n}))$ ($p \in V \subset M$) be a holomorphic local coordinate system.  
Let $v_{k} \in C^{\infty}(V)$ such that $t_{k} = v_{k} \xi^k s^k_{0}$.  
By Lemma~15.1.8 of \cite{Hor},  
we have 
\[
|v_{k}\xi^k(p)|^2 \lesssim k^{-2} \sup_{B(p, k^{-1})} |\overline{\partial}_{z}(v_{k}\xi^k)|^2 + k^{2n} \int_{B(p, k^{-1})} |v_{k}\xi^k|^2 d\mu_{\mathrm{Leb}}.  
\]
Here $B(p, r)$ is the Euclidean ball of center $p$ and radius $r$, and $d\mu_{\mathrm{Leb}}$ is the Lebesgue measure on $V$.  
We have $\overline{\partial}_{z} (v_{k}\xi^k) = \overline{\partial}_{z} (\chi_{k}\xi^k)$.  
Let $\eta = \sup_{B(p, k^{-1})} |\varphi_{0}(p)-\varphi_{0}(z)|$.  
We have $\eta = O(k^{-1})$.  
Then 
\begin{align*}
|t_{k}(p)|^2_{h^k} & \lesssim k^{-2} e^{k\eta} \sup_{B(p, k^{-1})} (|\overline{\partial}\chi_{k} \xi^k|^2 + k^2|\chi_{k}\xi^{k-1}\overline{\partial}\xi|^2)e^{-k\varphi_{0}} + k^{2n}e^{k\eta} \int_{B(p, k^{-1})} |t_{k}|^2_{h^k} \omega_n \\
& \lesssim
k^{-2}  \sup_{B(p, k^{-1})} (|\overline{\partial}\chi_{k}|^2 + k^2|\chi_{k}\overline{\partial}\xi|^2)e^{-k\varphi} +  O(k^{2+2n-m}).   
\end{align*}
It follows that 
\[
k^{-2} \sup_{B(p,k^{-1})} |\overline{\partial} \chi_{k}|^2e^{-k\varphi} \lesssim \frac{1}{k(\log k)^2} \sup_{B(p, k^{-1})} \large|\chi'(\frac{\sqrt{k}d_{X}}{\log k})\large|^2 e^{-k\pi (d_{X})^2} 
= O(k^{-\infty}), 
\]
\[
|\chi_{k}\overline{\partial}\xi|^2 e^{-k\varphi} \lesssim \chi_{k} (d_{X})^{2m}e^{-k\pi (d_{X})^2} = O(k^{-m}).  
\]
Since $X$ is compact, the above estimates do not depend on $p$ and the lemma is proved.  
\end{proof}

If we take $m > 2 + 2n$, we have $\lim_{k \to +\infty}|\alpha_{k}|^2_{h^k} = 1$ uniformly on $X$.  
Hence we obtain 
\[
\liminf_{k \to +\infty} \left(\frac{\mathrm{Vol}(X, \omega)}{(2k)^{n/2}} \sup_{f \in H^0_{(2), A}(M, L^k), f \neq 0}\frac{\inf_{x \in X} |f(x)|^2_{h^k}}{\|f\|^2_{h^k}}\right) \geq 1.  
\]
This completes the proof of Theorem~\ref{theorem:3}.

\vspace{5mm}

\par\noindent{\scshape \small
Department of Mathematics, \\
Ochanomizu University,  \\
2-1-1 Otsuka, Bunkyo-ku, Tokyo (Japan) }
\par\noindent{\ttfamily chiba.yusaku@ocha.ac.jp}
\end{document}